\documentclass[12pt]{article}
\usepackage[margin=1in]{geometry}
\usepackage{amsmath,amssymb,amsthm}
\usepackage{color}
\usepackage{parskip}
\usepackage{setspace}

\newtheorem{theorem}{Theorem}[section]

\newtheorem{corollary}[theorem]{Corollary}
\newtheorem{lemma}[theorem]{Lemma}
\newtheorem*{definition*}{Definition}
\def\F{\mathbb{F}}

\def\U{\mathcal{U}}
\def\V{\mathcal{V}}
\def\W{\mathcal{W}}
\def\RR{\mathcal{R}}
\def\S{\mathcal{S}}
\def\T{\mathcal{T}}
\newcommand{\I}{\mathcal{I}}

\begin{document}
\title{Exponential sum estimates over prime fields}
\author{
Doowon Koh\thanks{Department of Mathematics, Chungbuk National University. Supported by the National Research Foundation of Korea (NRF) grant funded by the Korea government (MIST) (No. NRF-2018R1D1A1B07044469). Email: {\tt koh131@chungbuk.ac.kr}}
\and 
Mozhgan Mirzaei \thanks{Department of Mathematics, University of California, San Diego. Partially supported by NSF grant DMS-1800746. Email: {\tt momirzae@ucsd.edu}}
\and 
Thang Pham\thanks{Department of Mathematics, University of California, San Diego. Supported by Swiss National Science Foundation grant P2ELP2175050. Email: {\tt v9pham@ucsd.edu}}
  \and
Chun-Yen Shen \thanks{Department of Mathematics, National Taiwan University. Partially supported by MOST, grant 104-2628-M-002-015-MY4. Email: {\tt cyshen@math.ntu.edu.tw}}
  }

\date{}
\maketitle
\date{}
\maketitle
\begin{abstract}
In this paper, we prove some extensions of recent results given by Shkredov and Shparlinski on multiple character sums for some general families of polynomials over prime fields. The energies of polynomials in two and three variables are our main ingredients. 
\end{abstract}

\section{Introduction}
Let $\mathbb{F}_p$ be a prime field, and $\chi$ be a non-trivial multiplicative character of $\mathbb{F}_p^*.$ Let $\delta>0$ be a real number. The Paley graph conjecture states that for any two sets $A, B\subset \mathbb{F}_p$ with $|A|, |B|>p^\delta,$ there exists $\gamma=\gamma(\delta)$ such that the following estimate holds: 
\begin{equation}\label{eq1}\left\vert \sum_{a\in A, b\in B}\chi(a+b)\right\vert<p^{-\gamma}|A||B|,\end{equation}
for any sufficiently large prime $p$ and any non-trivial character $\chi.$

If $|A|>p^{\frac{1}{2}+\delta}$ and $|B|>p^\delta,$ the conjecture has been confirmed by Karatsuba in \cite{k1, k2, k3}. In other ranges, the conjecture remains widely open, even in the balance case $|A|=|B|\sim p^{1/2}.$

In \cite{chang}, it is shown that if we have a restricted condition on the size of the sumset $B+B,$ then the inequality (\ref{eq1}) is true. The precise statement is as follows. 
\bigskip
\begin{theorem}[\cite{chang}]
Let $\delta$ and $K$ be positive numbers. Let $A, B$ be sets in $\mathbb{F}_p^*$ with $p>p(\delta, K)$ large enough and $\chi$ a non-trivial multiplicative character of $\mathbb{F}_p^*.$ Suppose that 
\[|A|>p^{\frac{4}{9}+\delta}, \]
\[|B|>p^{\frac{4}{9}+\delta},\]
\[|B+B|<K|B|.\]
Then there exists $\gamma=\gamma(\delta, K)>0$ such that 
\[\left\vert \sum_{a\in A, b\in B}\chi(a+b)\right\vert<p^{-\gamma}|A||B|.\]
\end{theorem}
\bigskip
In a recent work, Shkredov and Volostnov \cite{vol} improved this theorem in the case $A=B$ using a Croot-Sisask lemma on almost periodicity of convolutions of characteristic functions of sets \cite{croot}. For the sake of completeness, we will state their result in a general form as follows. 
\bigskip
\begin{theorem} [\cite{vol}]
Let $\delta,$ $K$ and $L$ be positive numbers. Let $A, B$ be sets in $\mathbb{F}_p^*$ with $p>p(\delta, K, L)$ large enough and $\chi$ a non-trivial multiplicative character of $\mathbb{F}_p^*.$ Suppose that 
\[|A|>p^{\frac{12}{31}+\delta}, \]
\[|B|>p^{\frac{12}{31}+\delta},\]
\[|A+A|<K|A|,\]
\[|A+B|<L|B|.\]
%Then there exists $\gamma=\gamma(\delta, K)>0$ such that 
Then we have
\[\left\vert \sum_{a\in A, b\in B}\chi(a+b)\right\vert<\sqrt{\frac{L\log 2K}{\delta \log p}}|A||B|.\]
\end{theorem}
Using recent advances in additive combinatorics, it has been indicated by Shkredov and Shparlinski \cite{shpar} that if we study the sums with more variables, then the problem becomes much easier. Namely, given four sets $\T, \U, \V, \W$ in $\mathbb{F}_p^*$ and two sequences of weights $\alpha=(\alpha_t)_{t\in \T},$ $\beta=(\beta_{u, v, w})_{u, v, w\in \U\times \V\times \W}$ with 
\[\max_{t\in \T}|\alpha_t|\le 1, ~\max_{(u, v, w)\in \U\times \V\times \W}|\beta_{uvw}|\le 1,\]
 they considered the following sum
\[S_\chi(\T, \U, \V, \W, \alpha, \beta, f):=\sum_{t\in \T, u\in \U, v\in \V, w\in \W}\alpha_t\beta_{uvw}\chi(t+f(u, v, w)),\]
where $f(x, y, z)$ is a polynomial in three variables in $\mathbb{F}_p[x, y, z].$ \\

Throughout this paper, we denote the cardinality of $\T, \U, \V, \W \subset \mathbb F_p$ by $T, U, V,W,$ respectively. We use $X\ll Y$ if $X\le C Y$ for some constant $C>0$ independent of the parameters related to $X$ and $Y,$ and  write $X\gg Y$ for $Y\ll X.$ The notation $X\sim Y$ means that both $X\ll Y$ and $Y\ll X$ hold. In addition, we use $X\lesssim Y$ to indicate that $X\ll (\log{Y}) Y.$

For the specific cases $f(x, y, z)=x+yz$ and  $f(x, y, z)=x(y+z),$ Shkredov and Shparlinski \cite{shpar} deduced the following result. 
\bigskip

\begin{theorem}[\cite{shpar}]\label{thmx}
For $\mathcal{U}, \mathcal{V}, \mathcal{W}, \T\subset \mathbb{F}_p^*,$ let $M=\max\{U, V, W\}.$ If $f(x, y, z)=x+yz$ or $f(x, y, z)=x(y+z),$ then for any fixed integer $n\ge 1,$ we have 
\begin{align*}
&\left\vert S_\chi(\T, \U, \V, \W, \alpha, \beta, f)\right\vert \ll 
\left((UVW)^{1-\frac{1}{4n}}+M^{\frac{1}{2n}}(UVW)^{1-\frac{1}{2n}}\right)\cdot \begin{cases} T^{\frac{1}{2}}p^{\frac{1}{2}} ~~\mbox{if}~~ n=1\\ Tp^{\frac{1}{4n}}+T^{\frac{1}{2}}p^{\frac{1}{2n}} ~~\mbox{if}~~n \ge 2.\end{cases}
\end{align*}
\end{theorem}

We note that this theorem is an improvement of the work of Hanson \cite{ha}. In order to indicate the strength of Theorem \ref{thmx}, the following interesting cases were considered by Shkredov and Shparlinski \cite{shpar}. 
\begin{enumerate}
\item If $U\sim V\sim W\sim T\sim N,$ then by setting $n=1,$ we have  
\[ \left|S_\chi(\T, \U, \V, \W, \alpha, \beta, f)\right|\ll N^{\frac{11}{4}}p^{\frac{1}{2}},\]
which is non-trivial whenever $N\ge p^{\frac{2}{5}+\epsilon}$ for some $\epsilon >0.$
\item Suppose that $T\ge p^\epsilon$ for some $\epsilon>0$ and $U\sim V\sim W\sim N.$ Taking $n=\lfloor\frac{2}{\epsilon}\rfloor+1,$ we have 
\[\left|S_\chi(\T, \U, \V, \W, \alpha, \beta, f)\right|\ll N^{3-\frac{3}{4n}} T p^{\frac{1}{4n}},\]
which is non-trivial as long as $N\ge p^{\frac{1}{3}+\delta}$ for some $\delta >0.$
\end{enumerate}
One can see \cite{4, 5, 7, 8, ha, alex, vol, m1, m2, s-0} and references therein for related results.

\subsection{Statement of main results}
The main purpose of this paper is to extend Theorem \ref{thmx} to a general form. More precisely, we consider any quadratic polynomial $f(x, y, z)$ which is not in the form of $g(h(x)+k(y)+l(z))$ for some polynomials $g, h, k, l$ in one variable.  We will also study the case of polynomials $f$ in two variables. Our first result is as follows. 
\bigskip
\begin{theorem}\label{thm2}
Let $f\in \F_p[x,y,z]$ be a quadratic polynomial that depends on each variable and that does not have the form $g(h(x)+k(y)+l(z)).$ For $\mathcal{U}, \mathcal{V}, \mathcal{W}\subset \mathbb{F}_p^*,$ let $\Omega=\max\{U^{-1}, V^{-1}, W^{-1}\}$ and let $\T\subset \mathbb{F}_p^*.$ Then the following statements hold:
\begin{enumerate}
\item If $UVW\ll p^2,$ then we have 
\begin{equation*}
\left\vert S_\chi(\T, \U, \V, \W, \alpha, \beta, f)\right\vert \ll 
\left((UVW)^{1-\frac{1}{4n}}+ UVW \Omega^{\frac{1}{n}}\right)\cdot \begin{cases} T^{\frac{1}{2}}p^{\frac{1}{2}} ~~\mbox{if}~~ n=1\\ Tp^{\frac{1}{4n}}+T^{\frac{1}{2}}p^{\frac{1}{2n}} ~~\mbox{if}~~n \ge 2.\end{cases}
\end{equation*}
\item If $UVW\gg p^2,$ then we have 
\[|S_\chi(\T, \U, \V, \W, \alpha, \beta, f)|\ll \left(\frac{UVW}{p^{1/2n}}+UVW \Omega^{\frac{1}{n}}\right)\cdot \begin{cases}T^{\frac{1}{2}}p^{\frac{1}{2}} ~~\mbox{if}~~ n=1\\ Tp^{\frac{1}{4n}}+T^{\frac{1}{2}}p^{\frac{1}{2n}} ~~\mbox{if}~~n \ge 2.\end{cases}\]
\end{enumerate}
\end{theorem}
\bigskip

As an immediate consequence of Theorem  \ref{thm2}, we get the following corollaries.
\bigskip

\begin{corollary}
Let $f\in \F_p[x,y,z]$ be a quadratic polynomial defined in Theorem \ref{thm2}.
%that depends on each variable and that does not have the form $g(h(x)+k(y)+l(z))$.
Let $\mathcal{U}, \mathcal{V}, \mathcal{W}, \T\subset \mathbb{F}_p^*$ such that $U\sim V\sim W \sim N$ and $T\ge p^\epsilon$ for some $\epsilon>0.$ Then the following statements hold:
\begin{enumerate}
%\item If $p^{\frac{1}{3}+\epsilon'}\ll N\ll p^{\frac{2}{3}}$ for some $\epsilon'>0$ and $n\ge \lfloor1/2\epsilon\rfloor+1,$ then we have 
\item If  $p^{\frac{1}{3}+\delta} \ll N\ll p^{\frac{2}{3}}$ for some $\delta>0$  and $n> \lfloor\frac{1}{2\epsilon}\rfloor+1,$ then we have 
\begin{align*}
&\left\vert S_\chi(\T, \U, \V, \W, \alpha, \beta, f)\right\vert \ll N^{3-\frac{3}{4n}} T p^{\frac{1}{4n}}.
\end{align*}
\item If  $N\gg p^{\frac{2}{3}}$ and $n> \lfloor \frac{1}{2\epsilon}\rfloor+1,$ then we have \[|S_\chi(\T, \U, \V, \W, \alpha, \beta, f)|\ll  \frac{N^3T}{p^{1/4n}}.\]
\end{enumerate}
\end{corollary}
\bigskip
\begin{corollary}
Let $f\in \F_p[x,y,z]$ be a quadratic polynomial defined in Theorem \ref{thm2}. For $\mathcal{U}, \mathcal{V}, \mathcal{W}, \T\subset \mathbb{F}_p^*$ with $U\sim V\sim W\sim T\sim N,$ we have the following conclusions:
\begin{enumerate}
\item Suppose that $p^{\frac{2}{5}+\delta} \ll N\ll p^{\frac{2}{3}}$ for some $\delta>0,$ then we have 
\begin{align*}
&\left\vert S_\chi(\T, \U, \V, \W, \alpha, \beta, f)\right\vert \ll N^{11/4}p^{1/2} ~(n=1).
\end{align*}
\item Suppose that $N\gg p^{2/3},$ then we have \[|S_\chi(\T, \U, \V, \W, \alpha, \beta, f)|\ll N^{7/2} ~(n=1).\]
\end{enumerate}
\end{corollary}
\bigskip

Now we address the results for two variable quadratic polynomial $f \in \mathbb F_p[x,y].$ 
Let $\chi$ be a non-trivial multiplicative character of $\mathbb{F}_p^*.$ Given three sets $\T, \U, \V$ in $\mathbb{F}_p^*,$  a polynomial $f\in \mathbb F_p[x,y],$ and two sequences of weights $\alpha=(\alpha_t)_{t\in \T}, \beta=(\beta_{u, v})_{u, v\in \U\times \V}$ with 
\[\max_{t\in \T}|\alpha_t|\le 1, ~\max_{(u, v)\in \U\times \V}|\beta_{uv}|\le 1,\] we define 
\[S_\chi(\T, \U, \V, \alpha, \beta, f)=\sum_{t\in \T, u\in \U, v\in \V}\alpha_t\beta_{uv}\chi(t+f(u, v)).\]

We are interested in finding an upper bound of the sum $S_\chi(\T, \U, \V, \alpha, \beta, f).$
In particular, we deduce strong results on this problems in the case when  $f\in \mathbb F_p[x,y]$ is a quadratic polynomial which is not of the form $g(\alpha x + \beta y)$ for some polynomial $g$ in one variable.
Relating this problem for two variable polynomials to that of three variable polynomials, we are able to prove the following result for two variable polynomials.
\bigskip
\begin{theorem}\label{thm1}
Let $f\in \F_p[x,y]$ be a quadratic polynomial which depends on each variable and which does not take the form $g(a x+b y).$ Given $\U, \V, \T \subset \mathbb{F}_p^*$ with $|\U-\V|\sim kU$ for some parameter $k>0,$ the following two statements hold:
\begin{enumerate}
\item If $V^2|\U-\V|\ll p^2,$ then we have 
\[|S_\chi(\T, \U, \V, \alpha, \beta, f)|\lesssim
\left( k^{\frac{3}{4n}}  \cdot \frac{UV}{U^{1/4n}V^{1/2n}} +k^{\frac{1}{n}}\cdot \frac{UV}{V^{1/n}}\right)\cdot \begin{cases}T^{\frac{1}{2}}p^{\frac{1}{2}} ~~\mbox{if}~~n=1\\
Tp^{\frac{1}{4n}}+T^{\frac{1}{2}}p^{\frac{1}{2n}} ~~\mbox{if} ~~n\ge 2.\end{cases}\]
\item If $V^2|\U-\V|\gg p^2,$ then we have 
\[|S_\chi(\T, \U, \V, \alpha, \beta, f)|\lesssim \left(k^{\frac{1}{n}}\cdot \frac{UV}{p^{1/2n}}+k^{\frac{1}{n}}\cdot \frac{UV}{V^{1/n}}\right)\cdot \begin{cases}T^{\frac{1}{2}}p^{\frac{1}{2}} ~~\mbox{if}~~n=1\\
Tp^{\frac{1}{4n}}+T^{\frac{1}{2}}p^{\frac{1}{2n}} ~~\mbox{if} ~~n\ge 2.\end{cases}\]
\end{enumerate}
\end{theorem}

As a consequence of Theorem \ref{thm1} for $k=1,$ we have the following corollary.
\bigskip
\begin{corollary}
Let $f\in \F_p[x,y]$ be a quadratic polynomial defined as in Theorem \ref{thm1}.  Assume that $\U, \V, \T \subset \mathbb{F}_p^*$ with $|\U-\V|\sim U,$ $U\sim V\sim N,$ and $T\ge p^\epsilon$ for some $\epsilon>0.$
Then the following statements hold:
\begin{enumerate}
\item Suppose that $p^{\frac{1}{3}+\epsilon'}\ll N\ll p^{\frac{2}{3}}$ for some $\epsilon'>0$ and $n> \lfloor 1/2\epsilon\rfloor+1.$ Then we have 
\[|S_\chi(\T, \U, \V, \alpha, \beta, f)|\lesssim N^{2-\frac{3}{4n}} T p^{\frac{1}{4n}}.\]
\item Suppose that $N\gg p^{2/3}$ and $n> \lfloor 1/2\epsilon\rfloor+1.$ Then we have 
\[|S_\chi(\T, \U, \V, \alpha, \beta, f)|\lesssim \frac{N^2T}{p^{1/4n}}.\]
\end{enumerate}
\end{corollary}
\bigskip
The rest of this paper is organized as follows: in Section $2$ we prove Theorem \ref{thm2}, and in Section $3$ we present the proof of Theorem \ref{thm1}.

\section{Proof of Theorem \ref{thm2}}
The following result is our main step in the proof of Theorem \ref{thm2}. This is the unbalanced energy version of Theorem $1.1$ in \cite{thang}. 
\bigskip
\begin{theorem}\label{thm12}
Suppose that $f\in \F_p[x,y,z]$ is a quadratic polynomial which depends on each variable and which does not take the form $g(h(x)+k(y)+l(z)).$ For $\U, \V, \W\subset \mathbb{F}_p^*$ with $UVW\ll p^2,$ let $E$ be the number of tuples $(u, v, w, u', v', w')\in (\U\times \V\times \W)^2$ such that $f(u, v, w)=f(u', v', w').$ Then we have 
\[E\ll (UVW)^{3/2}+ \max\{V^2W^2, V^2U^2, U^2W^2\}.\]
\end{theorem}
\begin{proof}
Let $f(x,y,z)$ be a quadratic polynomial that is not of the form $g(h(x)+k(y)+l(z)).$ Then $f$ has at least one of the mixed terms $xy, yz, xz,$ because otherwise $f$ would be in the form of $h(x)+k(y)+l(z).$ Moreover, 
we may assume that $f$ does not have any constant term, because the value $E$ is independent of the constant term in $f(x,y,z).$ Therefore, we may assume that $f(x,y,z)=axy+bxz+cyz+ r(x)+s(y)+t(z)$ where one of $a,b,c\in \mathbb F_p$ is not zero, and $r,s,t$ are polynomials in one variable with degree at most two and no constant terms. Furthermore, from the symmetric property of $f(x,y,z)$ we only need to prove Theorem \ref{thm12} for the following three cases:

\textbf{Case 1:} $f(x,y,z)=axy+bxz+ r(x)+s(y)+t(z)$ with $a\ne 0$ and $\deg(t)=2.$\\
\textbf{Case 2:} $f(x,y,z)=axy+bxz+ r(x)+s(y)+t(z)$ with $a\ne 0$ and $\deg(t)=1.$\\
\textbf{Case 3:} $f(x,y,z)=axy+bxz+ r(x)+s(y)$ with $a, b\ne 0.$ \\
\textbf{Case 4:} $f(x,y,z)=axy+ bxz +cyz+ r(x) + s(y)+t(z)$ with $a,b,c\ne 0.$

Notice that if one or two of the three mixed terms does not appear in the polynomial $f(x,y,z)$ (i.e. \textbf{Case 1, 2} or \textbf{3}), then the statement of Theorem \ref{thm12} follows immediately from Lemma \ref{lem:expandingwithoutyz}, \ref{lem:expandingwithoutyz1} and \ref{missingcase}  below. On the other hand, if the polynomial $f(x,y,z)$ has all the three mixed terms (i.e. \textbf{Case 4}), then Theorem \ref{thm12} is a direct consequence of Lemma \ref{lem:expandingwithyz}. Hence, the proof of Theorem \ref{thm12} is complete if we have the following four lemmas whose proofs will be given in the subsection below.

\begin{lemma}\label{lem:expandingwithoutyz}
Let $f(x,y,z) = axy+ bxz + r(x) + s(y)+t(z)$
be a quadratic polynomial in $\mathbb{F}_p[x, y, z]$ that depends on each variable with $a\ne 0$ and $\deg(t)=2.$ If $\U, \V,\W\subset \F_p^*$  with  $UVW\ll p^2,$ then we have 
%\[E\ll (UVW)^{3/2}+(U+V+W)(UVW)+V^2W^2,\]
\[E\ll (UVW)^{3/2}+\max\{U, V\} (UVW),\]
where $E$ denotes the number of tuples $(x, y, z, x', y', z')\in (\U\times \V\times \W)^2$ such that $f(x, y, z)=f(x', y', z').$
\end{lemma}
\bigskip
\begin{lemma}\label{lem:expandingwithoutyz1}
Let $f(x,y,z) = axy+ bxz + r(x) + s(y)+t(z)$
be a quadratic polynomial in $\mathbb{F}_p[x, y, z]$ that depends on each variable with $a\ne 0$ and $\deg(t)=1.$ Then for  $\U, \V,\W\subset \F_p^*$  with  $UVW\ll p^2,$ we have 
\[E\ll (UVW)^{3/2}+ \max\{V^2W^2, V^2U^2, U^2W^2\},\]
where $E$ is the number of tuples $(x, y, z, x', y', z')\in (\U\times \V\times \W)^2$ such that $f(x, y, z)=f(x', y', z').$
\end{lemma}
\bigskip

\begin{lemma}\label{missingcase}
Let $f(x,y,z) = axy+ bxz + r(x) + s(y)$
be a quadratic polynomial in $\mathbb{F}_p[x, y, z]$ that depends on each variable with $a, b\ne 0.$ Then for  $\U, \V,\W\subset \F_p^*$  with  $UVW\ll p^2,$ we have 
\[E\ll (UVW)^{3/2}+\max\{U, V\} (UVW),\]
where $E$ is the number of tuples $(x, y, z, x', y', z')\in (\U\times \V\times \W)^2$ such that $f(x, y, z)=f(x', y', z').$
\end{lemma}
\bigskip

\begin{lemma}\label{lem:expandingwithyz}
Let $f(x,y,z) = axy+ bxz +cyz+ r(x) + s(y)+t(z)$ be a quadratic polynomial in $\mathbb{F}_p[x, y, z]$ with $a, b, c\ne 0$ which depends on each variable and which does not take the form $g(h(x)+k(y)+l(z)).$ If  $\U, \V, \W\subset \F_p^*$ with $UVW\ll p^2,$ then  
\begin{align*}
E\ll (UVW)^{3/2}+ \max\{V^2W^2, V^2U^2, U^2W^2\},
\end{align*}
where  $E$ denotes the number of tuples $(x, y, z, x', y', z')\in (\U\times \V\times \W)^2$ such that $f(x, y, z)=f(x', y', z').$ 
\end{lemma}
\end{proof}

\subsection*{Proofs of Lemmas \ref{lem:expandingwithoutyz}, \ref{lem:expandingwithoutyz1}, \ref{missingcase}, and \ref{lem:expandingwithyz}}
In order to estimate the energy $E$ given in four lemmas above, we use 
the point-plane incidence bound due to Rudnev \cite{R}. A short proof can be found in \cite{frank}.
\bigskip
\begin{theorem}[Rudnev]\label{thm:rudnev}
Let $\RR, \S$ denote a set of points in $\F_p^3$ and a set of planes in $\mathbb{F}_p^3,$ respectively. Suppose that $|\RR|\ll |\S|$ and $|\RR|\ll p^2.$ In addition, assume that  there is no line that contains $k$ points of $\RR$ and is contained in $k$ planes of $\S.$
Then we have
\[ \I(\RR,\S):=|\{(p,\pi): p\in \RR, \pi\in  \S\}| \ll |\RR|^{1/2}|\S| +k|\S|.\] 
\end{theorem}

We also need the following Lemma.
\bigskip
\begin{lemma}[K\H ovari--S\'os--Tur\'an theorem, \cite{bo}]\label{lmx9}
Let $G=(A\cup B, E(G))$ be a $K_{2,t}$-free bipartite graph. Then the number of edges between $A$ and $B$ is bounded by 
 \[|E(G)|\ll t^{1/2}|A||B|^{1/2}+|B|.\]
\end{lemma}
\paragraph{Proof of Lemma \ref{lem:expandingwithoutyz}}
Let $E$ be the number of tuples $(x,y,z,x',y',z')\in (\U\times \V\times \W)^2$ such that
$ f(x,y,z) = f(x',y',z'),$ where the quadratic polynomial $f$ takes the form in \textbf{Case 1}. 
This implies that
\[ ayx- ax'y' +(bxz +r(x) + t(z)- s(y'))
= bx'z'+r(x')+t(z') - s(y).\]
This relation can be viewed as an incidence between the point $(x, y', bxz +r(x)+ t(z)- s(y'))$ in $\mathbb{F}_p^3$ and the plane defined by $ay X -ax'Y  +Z = bx'z'+r(x')+t(z') - s(y).$ Let $\RR$ be the following point set:
\[\RR: = \{(x, y', bxz +r(x)+ t(z)- s(y')) : (x,y',z)\in \U\times \V\times \W \} \subset \mathbb{F}_p^3,\]
and $\S$ be the following plane set
\[\S:=\{ay X -ax'Y  +Z = bx'z'+r(x')+t(z') - s(y) :
(x',y,z')\in \U\times \V\times \W
\}.\]
For each fixed $(u,v,w)\in \RR,$ at most two elements $(x,y',z)$ in $\U\times \V\times \W$ reproduce the $(u,v,w),$ because $\deg(t)=2.$ In fact, we can take $x=u, y'=v,$ and $z$ values are  solutions to 
$$ t(z)+buz+r(u)-s(v)=w.$$
By the same argument, we see that each fixed plane in $\S$ can be determined by at most two elements $(x',y,z')\in \U\times \V \times \W.$ Also notice that each element in $\U\times \V \times \W$ determines a point in $\RR$ and  a plane in $\S.$ Hence, we have that 
$$ |\RR| \sim  |\S| \sim  UVW \quad \mbox{and}\quad E\sim \I(\RR,\S).$$

%It is not hard to see that  $E$ is similar to $\I(\RR,\S)$, the number of incidences between $\RR$ and $\S$. Therefore, we can apply Theorem \ref{thm:rudnev} for $\RR$ and $\S$. To this end, we first need to estimate the sizes of $\RR$ and $\S.$ 
%Let $(u, v, w)$ be a point in $\RR$, we now count the number of triples $(x, y', z)\in \U\times \V\times \W$ such that 
%\[u=x, ~y'=v, w=bxz +r(x)+ t(z)- s(y').\]
%Since $\deg(t)=2$, the equation $t(z)+buz+r(u)-s(v)-w=0$ has at most two solutions $z$. In short, we can say that $|\RR|\sim UVW$. Using the same argument, we can conclude that $|\S|\sim UVW$. 
This shows that our problem is reducing to estimate of $\I(\RR,\S).$ To bound this, we apply Rudnev's point-plane incidence theorem. 
Since $|\RR|\sim UVW,$  the condition $|\RR|\ll p^2$ in Theorem \ref{thm:rudnev} is clearly satisfied from our assumption that $ UVW \ll p^2.$ 
Now, we count the number of collinear points in $\RR.$ Let $\RR'$ be the projection of $\RR$ onto the first two coordinates. It is clear that $\RR'=\U\times \V.$ Thus any line contains at most $\max\{U, V\}$ points unless it is vertical. In the case of vertical lines, we can see that no plane in $\S$ contains such lines, because the $z$-coordinate of normal vectors of planes in $\S$ is one. Therefore, we can apply Theorem \ref{thm:rudnev} with $k=\max\{U, V\}.$ In other words, we obtain 
\[E\ll (UVW)^{3/2}+ \max\{U, V\} (UVW).\]
This completes the proof of Lemma \ref{lem:expandingwithoutyz}. $\square$. 

\paragraph{Proof of Lemma \ref{lem:expandingwithoutyz1}}
Since $\deg(t)=1,$ without loss of generality, we assume that $t(z)=mz$ for some $m\in \mathbb{F}_p^*$ and so
$f(x,y,z) = axy+ bxz + r(x) + s(y)+mz.$ As in the proof of Lemma \ref{lem:expandingwithoutyz}, we define 
the set $\RR$ of points and the set $\S$ of planes as follows:

\[\RR: = \{(x, y', bxz +r(x)+ mz- s(y')) : (x,y',z)\in \U\times \V\times \W \} \subset \mathbb{F}_p^3,\]
\[\S:=\{ay X -ax'Y  +Z = bx'z'+r(x')+mz' - s(y) : (x',y,z')\in \U\times \V\times \W\}.\]
The only reason we need to prove Lemma \ref{lem:expandingwithoutyz1} is that if $u=-m/b\in \U,$ then the triples $(-m/b, v, w)\in \RR$ can be determined by many triples $(x, y', z)\in \U\times \V\times \W.$ For this case, we need to do some more technical steps. If $-m/b\not\in \U,$ then  Lemma \ref{lem:expandingwithoutyz1} follows immediately from the same argument as in the proof of Lemma \ref{lem:expandingwithoutyz}.
Thus we may assume that $u=-m/b\in \U.$ As above, we first need to estimate the sizes of $\RR$ and $\S.$ For $(u,v,w)\in \RR $ and $(x, y',z)\in \U\times \V \times \W,$ we consider the following system of three equations:
\[u=x, ~v=y', ~w= buz +r(u)+ mz- s(v).\]
If $u\in \U$ satisfies $bu=-m,$ i.e. $u=-m/b\in \U,$ then  we have
\begin{equation}\label{ShenEq} u=x, ~v=y', ~w=r(u)-s(v) \quad \mbox{for all}~~ z\in \W.\end{equation}
Let $\RR_1$ be the set of points $(u, v, w)\in \RR$ with $u=-m/b.$ Then $\RR_1$ is a set with $V$ points, since for any $v=y'\in \V,$ $w$ is determined uniquely.  By \eqref{ShenEq} and the definition of $\RR_1,$ notice that each point in $\RR_1$ is determined by $W$ triples $(x, y', z)\in \U\times \V\times \W.$ Let $\RR_2=\RR\setminus \RR_1.$ Also notice that each point in $\RR_2$ is determined by exactly one triple $(x, y', z)\in \U\times \V\times \W.$

By the similar argument, we can partition the set of planes $\S$ into two sets $\S_1$ and $\S_2$ with $\S_2=\S\setminus \S_1$ so that  $|\S_1|=V,$  each plane in $\S_1$ is determined by $W$ triples $(x', y, z')\in \U\times \V\times \W,$ and each plane in $\S_2$ is determined by exactly one triple $(x', y, z')\in \U\times \V\times \W.$ 

From the above observations, it follows that each incidence between $\RR_1$ and $\S_2,$ or between $\RR_2$ and $\S_1$ contributes to $E$ by $W,$ each incidence between $\RR_1$ and $\S_1$ contributes to $E$ by $W^2,$ and each incidence between $\RR_2$ and $\S_2$ contributes to $E$ by one. Namely, we have
$$  E \ll W^2\cdot \I(\RR_1,\S_1) + W\cdot \I(\RR_1,\S_2) + W\cdot \I(\RR_2,\S_1) +\I(\RR_2,\S_2).$$ 

Since $|\RR_1|=|\S_1|=V,$ it is clear that
\[I(\RR_1, \S_1)\ll  V^2.\]

To bound $I(\RR_2, \S_2),$ recall that each element of $\RR_2$ and $\S_2$ is determined by exactly one element  $(x,y,z) \in \U\times \V \times W$ with $x\ne -m/b.$ Hence, by the same argument as in the proof of Lemma \ref{lem:expandingwithoutyz}, we see that 
\[I(\RR_2, \S_2)\ll (UVW)^{3/2}+ \max\{U, V\} (UVW).\]
To bound $I(\RR_1, \S_2),$ we will use Lemma \ref{lmx9}.  
Let $G$ denote the bipartite graph with vertex sets $\S_2$ and $\RR_1$ such that there is an edge between a point in $\RR_1$ and  a plane in $\S_2$ if the point lies on the plane.
Since $|\RR_1|=V,$ each line contains at most $V$ points in $\RR_1,$ and so any two planes in $\S_2$ support at most $V$ points in common. Thus letting $A:=\RR_1$ and $B:=\S_2$ and applying Lemma \ref{lmx9}, we obtain that
\[I(\RR_1, \S_2)=|E(G)|\ll V^{1/2}V(UVW)^{1/2}+UVW = U^{1/2}W^{1/2}V^2+UVW.\]
Similarly, we also have 
\[I(\RR_2, \S_1)\ll U^{1/2}W^{1/2}V^2+UVW.\]
In other words, we have proved that 
\begin{align*}
E&\ll (UVW)^{3/2}+ \max\{U, V, W\} (UVW) + V^2W^2 + U^{1/2}V^2W^{3/2}\\&\ll (UVW)^{3/2}+ V^2W^2 + V^2 U^2 + U^2 W^2
\\&\ll(UVW)^{3/2}+ \max\{V^2W^2, V^2 U^2, U^2 W^2\}.
\end{align*}

%\begin{align*}
%E&\ll (UVW)^{3/2}+UV^2W+U^2VW+UVW^2+V^2W^2+U^{1/2}V^2W^{3/2}\\&\ll (UVW)^{3/2}+UV^2W+U^2VW+UVW^2+V^2W^2.
%\end{align*}
This completes the proof of Lemma \ref{lem:expandingwithoutyz1}. $\square$.

\paragraph{Proof of Lemma \ref{missingcase}:}
Since $f(x,y,z) = axy+ bxz + r(x) + s(y)$ with $a, b\ne 0,$
as in the proof of Lemma \ref{lem:expandingwithoutyz}, we can define 
the set $\RR$ of points and the set $\S$ of planes as follows:

\[\RR: = \{(x, y', bxz +r(x)- s(y')) : (x,y',z)\in \U\times \V\times \W \} \subset \mathbb{F}_p^3,\]
\[\S:=\{ay X -ax'Y  +Z = bx'z'+r(x')- s(y) : (x',y,z')\in \U\times \V\times \W\}.\]
Since $b\ne 0$ and $\U\subset \mathbb F_p^*,$  we have
$$|\RR|=|\S|=UVW \quad \mbox{and} \quad E=\I(\RR,\S).$$
By the same argument as in the proof of Lemma \ref{lem:expandingwithoutyz}, we conclude that 
$$E\ll (UVW)^{3/2}+\max\{U, V\} (UVW),$$ 
as desired.
$\square$.

\paragraph{Proof of Lemma \ref{lem:expandingwithyz}:}
Now we would like to estimate $E$ which is the number of tuples $(x,y,z,x',y',z') \in (\U\times \V\times \W)^2$ satisfying the equation
\begin{equation}\label{saE} f(x,y,z)=f(x', y',z'),\end{equation}
where $f(x,y,z) = axy+ bxz +cyz+ r(x) + s(y)+t(z)$ is a quadratic polynomial in $\mathbb{F}_p[x, y, z]$ with $a, b, c\ne 0.$ Without loss of generality, we may assume that 
$$f(x,y,z)= axy+ bxz +cyz + dx^2+ey^2+gz^2 +hx +iy+jz,$$
where $a,b,c\ne 0$ and $d,e,g,h, i, j\in \mathbb F_q.$ 
We adapt the argument as in the proof of Lemma 2.3 in \cite{thang}.
%%%%%%%%%%%%%%%%%%%%%%%%%%%%%%%%%%%%%%%%%%%%%%%%%%%%%%%%%%%%%%%5
Since the polynomial $f(x,y,z)$ is not in the form of $g(h(x)+k(y)+l(z)),$  one of the following equations  is not satisfied:
\begin{equation*}
4de=a^2,~ 4dg =b^2,~ 4eg =c^2,~  hc=ja=ib.
\end{equation*}

 Otherwise, we could write 
$$f =\left(\sqrt{d} x + \sqrt{e}y + \sqrt{g}z + \frac{h}{2\sqrt{d}}\right)^2 -\frac{h^2}{4d},$$ 
if all of $d, e, g$ are squares in $\mathbb F_q$. On the other hand, if all of $d,e,g$ are not squares in $\mathbb F_q$, we could write 
$$f =\frac{1}{deg} \left( d\sqrt{eg} x+ e\sqrt{dg}y + g\sqrt{de}z + \frac{h\sqrt{eg}}{2}\right)^2-\frac{h^2}{4d},$$
since the equations $4de=a^2, 4dg =b^2, 4eg =c^2$ imply that $de, dg, eg$ are squares in $\mathbb F_q,$ and $e, d, g$ are nonzeros.

By permuting the variables, we may assume that one of the following equations does not hold:
\begin{equation*}4eg=c^2,~ ib=ja. \end{equation*}

%%%%%%%%%%%%%%%%%%%%%%%%%%%%%%%%%%%%%%%%%%%%%%%%%%%%%%%%%%%%%%%%%%%%
The equation \eqref{saE} is rewritten as
%It follows from the equation $f(x,y,z) = f(x',y',z')$ that
\begin{align*}
(ay+bz)x  -x' (ay'+bz')&+  d x^2 - e(y')^2 -cy'z'- g(z')^2+hx-iy'-jz' \\
&= 
d(x')^2-e y^2 -cyz -g z^2 +hx' -iy-jz .
\end{align*}
This relation can be viewed as an incidence between the point $(x, ay'+bz', 
 dx^2 - e (y')^2-cy'z'- g(z')^2 +hx-iy'-jz')$ in $\mathbb{F}_p^3$ and the plane defined by 
\[(ay+bz)X  -x' Y+ Z = 
d (x')^2-e y^2 -cyz- g z^2 +hx'-iy-jz.\]
Let $\RR$ be the following set of points
\[\RR = \{(x, ay'+bz', 
 dx^2 - e (y')^2-cy'z'- g(z')^2 +hx-iy'-jz') : (x,y',z')\in \U\times \V\times \W \} ,\]
and $\S$ be the following set of planes
\[ \S=\{ (ay+bz)X  -x' Y+ Z = 
d (x')^2-e y^2 -cyz- g z^2 +hx'-iy-jz :
(x',y,z)\in \U\times \V\times \W
\}.\]
It is clear that $E$ is bounded from above by the number of incidences between $\RR$ and $\S.$ In the next step, we estimate the sizes of $\RR$ and $\S.$ Indeed, for a given point $(u, v, w)\in \RR,$ we now count the number of triples $(x, y', z')\in \U\times \V\times \W$ such that 
\[u=x,~~~ v = ay'+bz', ~~~w = dx^2 - e (y')^2-cy'z'- g(z')^2 +hx-iy'-jz'.\]
These equations yield that
\[w = d u^2 - e (y')^2 -cy' \left(\frac{v-ay'}{b}\right) - g\left(\frac{v-ay'}{b}\right)^2 
+hu-iy'-j\left(\frac{v-ay'}{b}\right), \]
or
\[\left(b^2e-abc+a^2g\right)(y')^2  +\left(bcv - 2agv +ib^2 - jab\right)y' +\left(b^2w - b^2d u^2+g v^2 -b^2hu +bjv\right)= 0.\]
We consider the following two cases:

{\bf Case $1$:} If either $b^2e-abc+a^2g$ or $bcv-2agv+ib^2-jab$ is non-zero, then at most two solutions $y'$ of the above equation exist, and $z'$ value is determined in terms of $v$ and $y'.$ 

{\bf Case $2$:} If both $b^2e-abc+a^2g$ and $bcv-2agv+ib^2-jab$ are zero, then we will have the following system:
\begin{equation}\label{0911}
b^2e-abc+a^2g=0, ~~~
(bc - 2ag)v +(ib - ja)b=0,~~~
b^2w - b^2d u^2+g v^2 -b^2hu +bjv=0.\end{equation}
In this case, we need to do some more technical steps. 

In the case when $bc-2ag=0,$ the second equation above tells us that $ib=ja.$ Therefore, it follows from the first equation that $4eg=c^2,$ which contradicts  our assumptions at the beginning of the proof. 

Thus we must have $bc-2ag\ne 0.$ This gives us $v = -(ib^2-jab)/(bc-2ag).$ For this value of $v$ and any $u\in \U, w$ is determined uniquely by the third equation of (\ref{0911}). Therefore, there are at most $U$ points $(u, v, w)\in \RR$ which satisfy three equations above. We denote the set of these points by $\RR_2\subset \RR.$ Let $\RR_1=\RR\setminus \RR_2.$  We have $|\RR_2|=U$ and $|\RR_1|\sim UVW.$ Note that any point in $\RR_1$ corresponds to at most two points in $\U\times \V\times \W$ and any point in $\RR_2$ corresponds to at most $\max\{V, W\}$ points $(x,y',z')\in \U\times \V\times \W.$

Likewise, we can also show that the plane set $\S$ can be partitioned into two sets $\S_1$ and $\S_2,$ where each plane in $\S_1$ corresponds to at most two points in $\U\times \V\times \W,$ and each plane in $\S_2$ corresponds to at most $\max\{V, W\}$ points in $\U\times \V\times \W.$

Set $N:=\max\{V, W\}.$ We observe that an incidence between $\RR_2$ and  $\S_1,$ 
or between $\RR_1$ and $\S_2,$ contributes at most $N$ to $E,$ and an incidence between$\RR_2$ and $\S_2$ contributes at most $N^2$ to $E.$ Hence, we obtain
\begin{equation}\label{eq:split}
E \ll \I(\RR_1,\S_1) + N\cdot \I(\RR_1,\S_2) + N\cdot \I(\RR_2,\S_1) + N^2\cdot\I(\RR_2,\S_2).
\end{equation}
Since $|\RR_2|, |\S_2|\ll U,$ we have $\I(\RR_2,\S_2)\leq U^2.$ To bound $\I(\RR_1,\S_1),$ we will apply Theorem \ref{thm:rudnev}. Before doing this, we need to give an upper bound on the number of collinear points in $\RR.$

One can cover the set $\RR$ by $U$ planes defined by the equations $x=x_0, x_0\in U.$ By a direct computation, one can check that for each plane $x=x_0,$ the points of $\RR$ on this plane lie on either a line or a parabola, and for distinct $y'\in V,$ we have distinct parabolas or lines. 

If a line $l$ does not lie on any of those planes, then it intersects $\RR$ in at most $U$ points. Suppose that $l$ lies on the plane $x=x_0.$ Then there are two possibilities. If $l$ is the same as a line determined by some $y'\in V,$ then it contains $W$ points. If it is not that case, then $l$ supports at most $2V$ points from $\RR,$ since a line intersects a parabola or a line in at most two points. In other words, we can say that the maximal number of collinear points in $\RR$ is at most $U+2V+W.$ By Theorem \ref{thm:rudnev}, we have
\[\I(\RR_1,\S_1) \ll (UVW)^{3/2} + \max\{U, V, W\} (UVW). \]

To bound $I(\RR_1, \S_2)$ and  $I(\RR_2, \S_1),$  we use Lemma \ref{lmx9} again. Let $G$ be the bipartite graph with vertex sets $\S_2$ and $\RR_1$ such that there is an edge between a point and a plane if the point lies on the plane. We showed that no $\max\{U, V, W\}+1$ points of $\RR_1$ lie on a line. Hence, any two planes of $\S_2$ contain at most $\max\{U, V, W\}$ points of $\RR_1$ in common. Thus, we get
\[ \I(\RR_1,\S_2) = |E(G)| \ll (\max\{U, V, W\})^{1/2}\cdot U\cdot (UVW)^{1/2} + UVW .\]
Using a similar argument, we get
\[\I(\RR_2,\S_1)\ll (\max\{U, V, W\})^{1/2}\cdot U\cdot (UVW)^{1/2} + UVW.\]
Putting all bounds together, it follows from \eqref{eq:split} that
\begin{equation}\label{appdom}
E\ll (UVW)^{3/2} + M(UVW)+N M^{\frac{1}{2}} U^{\frac{3}{2}} V^{\frac{1}{2}} W^{\frac{1}{2}} + N (UVW) + N^2 U^2,
\end{equation}
where $N=\max\{V,W\}$ and $M=\max\{U, V, W\}.$
A direct computation shows that each of the second, third, fourth, and fifth terms in the RHS of the equation \eqref{appdom} is dominated by
$$ V^2W^2 + V^2 U^2 + U^2 W^2.$$ 
Hence, we have
 \begin{align*} E&\ll (UVW)^{3/2}+ V^2W^2 + V^2 U^2 + U^2 W^2\\
                 &\ll (UVW)^{3/2}+ \max\{V^2W^2, V^2 U^2, U^2 W^2\}, \end{align*}
which completes the proof of Lemma \ref{lem:expandingwithyz}. $\square$
%\begin{align*}
%E&\ll  U^2(V^2+W^2)+(UVW)^{3/2} + (V+W+U)(UVW)+(\max\{U, V, W\})^{3/2}U^{3/2}V^{1/2}W^{1/2}.
%\end{align*}

%Since the polynomial $f(x, y, z)$ is symmetric in $x, y, z,$ we can switch the roles of $\U, \V, \W$ in the sets $\RR$ and $\S$ if necessary. Therefore, we may assume that $U\le V\le W.$ This gives us 
%\[E\ll  (UVW)^{3/2} + (V+W+U)(UVW).\]
%which completes the proof of Lemma \ref{lem:expandingwithyz}. $\square$

In addition to Theorem \ref{thm12},  the following lemma also plays an important role in providing the complete proof of the first part of Theorem \ref{thm2}. 
\bigskip

\begin{lemma}[\cite{shpar}, Lemma $2.3$]\label{lm1754}
For $\T\subset\mathbb{F}_p^*$ with size $T$ and a sequence of weights $\alpha=(\alpha_t)_{t\in \T}$ with $\max_{t\in \T} |\alpha_t|\le 1,$ and for any fixed integer $n\ge 1,$ we have
\[ \sum_{\lambda\in \mathbb{F}_p}\left\vert \sum_{t\in \T}\alpha_t\chi(\lambda+t)\right\vert^{2n}\ll \begin{cases}Tp ~~\mbox{if}~~n=1\\ T^{2n}p^{1/2}+T^np ~~\mbox{if}~~n\ge 2.\end{cases}\]
\end{lemma}

To prove the second part of Theorem \ref{thm2}, we use following point-plane incidence theorem due to Vinh (\cite{vinh}). \\

\begin{theorem}[\cite{vinh}, Theorem 5]\label{VinhInc} Suppose that $\RR$ is a collection of points in $\mathbb F_q^d,$  and $\S$ is a collection of hyperplanes in $\mathbb F_q^d,$ with $d\ge 2.$ Then we have
$$ \I(\RR,\S):=|\{(p,\pi): p\in \RR, \pi\in  \S\}| \ll \frac{|\RR||\S|}{q}+ q^{(d-1)/2} |\RR|^{1/2} |\S|^{1/2}. $$

\end{theorem}

Using Theorem \ref{thm12} and the argument in \cite{shpar}, we are now ready to give a proof of Theorem \ref{thm2}. 

\paragraph{Proof of Theorem \ref{thm2}:} 
Since $\max_{(u,v, w)\in \U\times \V\times \W}|\beta_{uvw}|\le 1,$ we have 
\[|S_\chi(\T, \U, \V, \W, \alpha, \beta, f)|\le \sum_{u\in \U, v\in \V, w\in \W}\left\vert \sum_{t\in \T}\alpha_t\chi(t+f(u, v, w))\right\vert.\]
For $\lambda\in \mathbb{F}_p,$ let $N(\U, \V, \W, \lambda)$ be the number of solutions of the equation 
\[f(u, v, w)=\lambda,\]
with $(u, v, w)\in \U\times \V\times \W.$ One can check that 
\[\sum_{\lambda\in \mathbb{F}_p}N(\U, \V, \W, \lambda)=UVW,\]
and 
\[\sum_{\lambda\in \mathbb{F}_p}N(\U, \V, \W, \lambda)^2= E,\]
where $E$ is the number of tuples $(u, v, w, u', v', w')\in (\U\times \V\times \W)^2$ such that $f(u, v, w)=f(u', v', w').$

Thus we have 
\begin{align*}
|S_\chi(\T, \U, \V, \W, \alpha, \beta, f)|\le \sum_{\lambda\in \mathbb{F}_p}N(\U, \V, \W, \lambda)\left\vert \sum_{t\in \T}\alpha_t\chi(t+\lambda)\right\vert.
\end{align*}
Using the H\"{o}lder inequality, we have 
\begin{align*}
&|S_\chi(\T, \U, \V, \W, \alpha, \beta, f)|^{2n}\le \left(\sum_{\lambda\in \mathbb{F}_p}\left\vert \sum_{t\in \T}\alpha_t\chi(t+\lambda)\right\vert^{2n}\right)\cdot \left(\sum_{\lambda\in \mathbb{F}_p}N(\U, \V, \W, \lambda)^{\frac{2n}{2n-1}}\right)^{2n-1}\\
&\le \left(\sum_{\lambda\in \mathbb{F}_p}N(\U, \V, \W, \lambda)\right)^{2n-2}\cdot \left(\sum_{\lambda\in \mathbb{F}_p}N(\U, \V, \W, \lambda)^2\right)\cdot \left(\sum_{\lambda\in \mathbb{F}_p}\left\vert \sum_{t\in \T}\alpha_t\chi(t+\lambda)\right\vert^{2n}\right)\\
&= (UVW)^{2n-2}\cdot E\cdot \left(\sum_{\lambda\in \mathbb{F}_p}\left\vert \sum_{t\in \T}\alpha_t\chi(t+\lambda)\right\vert^{2n}\right).
\end{align*}
It follows from Theorem \ref{thm12} and Lemma \ref{lm1754} that if $UVW\ll p^2,$ then
\begin{align*}
&|S_\chi(\T, \U, \V, \W, \alpha, \beta, f)|\ll \\
& (UVW)^{\frac{2n-2}{2n}}\left( (UVW)^{\frac{3}{2}} + \max\{V^2W^2, V^2U^2, U^2W^2\} \right)^{\frac{1}{2n}} \cdot \begin{cases}T^{\frac{1}{2}}p^{\frac{1}{2}} ~~\mbox{if}~~ n=1\\ Tp^{\frac{1}{4n}}+T^{\frac{1}{2}}p^{\frac{1}{2n}} ~~\mbox{if}~~n \ge 2.\end{cases}\\
&\ll \left((UVW)^{1-\frac{1}{4n}}+ UVW \Omega^{\frac{1}{n}}\right)\cdot \begin{cases} T^{\frac{1}{2}}p^{\frac{1}{2}} ~~\mbox{if}~~ n=1\\ Tp^{\frac{1}{4n}}+T^{\frac{1}{2}}p^{\frac{1}{2n}} ~~\mbox{if}~~n \ge 2.\end{cases} \end{align*}
This completes the proof of the first part of Theorem  \ref{thm2}. 

Next we prove the second part of Theorem \ref{thm2}. Suppose that $UVW\gg p^2.$ Instead of Rudnev's point-plane incidence theorem (Theorem \ref{thm:rudnev}), one can follow the proof of Theorem \ref{thm12} with Vinh's point-plane incidence theorem (Theorem \ref{VinhInc}). Then we see that 
$$E\ll (UVW)^2/p+\max\{V^2W^2, V^2U^2, U^2W^2\}.$$ 
With this bound of $E,$ we have 
\[|S_\chi(\T, \U, \V, \W, \alpha, \beta, f)|\ll \left(\frac{UVW}{p^{1/2n}}+UVW \Omega^{\frac{1}{n}}\right)\cdot \begin{cases}T^{\frac{1}{2}}p^{\frac{1}{2}} ~~\mbox{if}~~ n=1\\ Tp^{\frac{1}{4n}}+T^{\frac{1}{2}}p^{\frac{1}{2n}} ~~\mbox{if}~~n \ge 2,\end{cases}\]
which completes the proof of the second part of Theorem  \ref{thm2}.
Thus the proof of Theorem \ref{thm2} is complete.  $ \square$
 
\section{Proof of Theorem \ref{thm1}}
In the proof of Theorem \ref{thm1}, we make use of the following result which can be obtained by applying Theorem \ref{thm12}.
\bigskip
\begin{theorem}\label{theorem21}
Let $f\in \F_p[x,y]$ be a quadratic polynomial that depends on each variable and that does not have the form $g(a x+b y).$ For $\U, \V \subset \mathbb{F}_p^*,$  let $E$ be the number of tuples $(u, v, u', v')\in (\U\times \V)^2$ such that $f(u, v)=f(u', v').$ Suppose that $V^2|\U-\V|\ll p^2.$ Then we have 
\[E\lesssim V|\U-\V|^{3/2}+|\U-\V|^2.\]
\end{theorem}
\begin{proof}
For any $t\in f(\U, \V),$ let $m_t$ be the number of pairs $(u, v)\in \U\times \V$ such that $f(u, v)=t.$ It is clear that $m_t \le UV$ for all $t\in f(\U, \V).$ It follows that
\begin{equation}\label{logeq} E=\sum_{t\in f(\U, \V)}m_{t}^2= \sum_{j}\sum_{t\in f(\U, \V),\\  2^j\le m_{t}<2^{j+1}}m_{t}^2\ll \sum_{j=0}^{\log(UV)}2^{2j+2}k_{2^j},\end{equation}
where $k_{2^j}$ denotes the cardinality of the set $D_j:=\{t\in  f(\U, \V) \colon m_t\ge 2^j\}.$
We now bound $k_{2^j}$ as follows. 

Let $h(x, y, z)=f(x-z, y).$ Since $f(x, y)$ is not of the form $g(a x+b y),$ by a direct computation, we have $h(x, y, z)$ satisfies the conditions of Theorem \ref{thm12}. We now consider the following equation 
\begin{equation}\label{eq12x} h(x, y, z)=t,\end{equation}
where $x\in \V, z\in \V-\U, y\in \V, t\in D_j\subset f(\U, \V).$ Let $N(h)$ be the number of solutions of this equation. It is easy to see that $N(h)\ge 2^jk_{2^j}V.$ By Cauchy-Schwarz inequality, we have 
\begin{align*}
N(h)&\ll k_{2^j}^{1/2}\left|\{(x, y, z, x', y', z')\in \left(\V\times \V \times (\V-\U)\right)^2\colon h(x, y, z)=h(x', y', z')\}\right|^{1/2}\\
&\ll k_{2^j}^{1/2}\left(V^{3/2}|\U-\V|^{3/4}+|\U-\V|V\right),
\end{align*}
where the second inequality follows from Theorem \ref{thm12} with the condition $V^2|\U-\V|\ll p^2.$ Putting the lower bound and the upper bound of $N(h)$ together we get
\[2^j k_{2^j}V\ll k_{2^j}^{1/2}\left(V^{3/2}|\U-\V|^{3/4}+|\U-\V|V\right).\]
This gives us 
\[k_{2^j}\ll \frac{V|\U-\V|^{3/2}+|\U-\V|^2}{2^{2j}}.\]
Combining this estimate with the inequality \eqref{logeq}, we see that
$$ E \ll \left(V|\U-\V|^{3/2}+|\U-\V|^2\right) \sum_{j=0}^{\log({UV})} 1 \lesssim V|\U-\V|^{3/2}+|\U-\V|^2.$$
%We note that $\sum_{j}j\cdot k_j\ll UV$. So $k_j\ll UV/j$. Set $\gamma=\frac{|\U-\V|^{3/2}}{U}+\frac{|\U-\V|^2}{UV}$. We now express $E$ as follows:
%\[E\ll 4\sum_{j=0}^{\log \gamma} 2^{2j}k_{2^j}+\sum_{\log \gamma +1}^{\log (|A||B|)}2^{2j}k_{2^j}\lesssim V|\U-\V|^{3/2}+|\U-\V|^2.\]
This concludes the proof of Theorem \ref{theorem21}. 
\end{proof}
\paragraph{Proof of Theorem \ref{thm1}:} The proof of Theorem \ref{thm1} is similar to Theorem \ref{thm2} except that we use Theorem \ref{theorem21} instead of Theorem \ref{thm12}. For the completeness, we will include the detailed proof here. 

Since $\max_{(u,v)\in \U\times \V}|\beta_{uv}|\le 1,$ we have 
\[|S_\chi(\T, \U, \V, \alpha, \beta, f)|\le \sum_{u\in \U, v\in \V, w\in \W}\left\vert \sum_{t\in \T}\alpha_t\chi(t+f(u, v))\right\vert.\]
For $\lambda\in \mathbb{F}_p,$ let $N(\U, \V, \lambda)$ be the number of solutions of the equation 
\[f(u, v)=\lambda,\]
with $(u, v)\in \U\times \V.$ It is easy to see
\[\sum_{\lambda\in \mathbb{F}_p}N(\U, \V, \lambda)=UV, \quad \mbox{and}\quad \sum_{\lambda\in \mathbb{F}_p}N(\U, \V, \lambda)^2= E,\]
where $E$ is defined as in Theorem \ref{theorem21}. Thus we have 
\begin{align*}
|S_\chi(\T, \U, \V, \alpha, \beta, f)|\le \sum_{\lambda\in \mathbb{F}_p}N(\U, \V, \lambda)\left\vert \sum_{t\in \T}\alpha_t\chi(t+\lambda)\right\vert.
\end{align*}
Using the H\"{o}lder inequality, we have 
\begin{align*}
&|S_\chi(\T, \U, \V, \alpha, \beta, f)|^{2n}\le \left(\sum_{\lambda\in \mathbb{F}_p}\left\vert \sum_{t\in \T}\alpha_t\chi(t+\lambda)\right\vert^{2n}\right)\cdot \left(\sum_{\lambda\in \mathbb{F}_p}N(\U, \V, \lambda)^{\frac{2n}{2n-1}}\right)^{2n-1}\\
&\ll \left(\sum_{\lambda\in \mathbb{F}_p}N(\U, \V, \lambda)\right)^{2n-2}\cdot \left(\sum_{\lambda\in \mathbb{F}_p}N(\U, \V, \lambda)^2\right)\cdot \left(\sum_{\lambda\in \mathbb{F}_p}\left\vert \sum_{t\in \T}\alpha_t\chi(t+\lambda)\right\vert^{2n}\right)\\
&= (UV)^{2n-2}\cdot E\cdot \left(\sum_{\lambda\in \mathbb{F}_p}\left\vert \sum_{t\in \T}\alpha_t\chi(t+\lambda)\right\vert^{2n}\right).
\end{align*}
By Theorem \ref{theorem21} and Lemma \ref{lm1754}, we see that if $V^2|\U-\V|\sim kUV^2\ll p^2,$ then
\begin{align*}
&|S_\chi(\T, \U, \V, \alpha, \beta, f)|\lesssim
&\left( k^{\frac{3}{4n}}  \cdot \frac{UV}{U^{1/4n}V^{1/2n}} +k^{\frac{1}{n}}\cdot \frac{UV}{V^{1/n}}\right)\cdot \begin{cases}T^{1/2}p^{1/2} ~~\mbox{if}~~n=1\\
Tp^{1/4n}+T^{1/2}p^{1/2n} ~~\mbox{if} ~~n\ge 2.\end{cases}
\end{align*}
This proves the first part of Theorem \ref{thm1}. 

To prove the second part of  Theorem \ref{thm1}, assume that $V^2|\U-\V|\gg p^2.$ We can follow the proof of Theorem \ref{theorem21} with Vinh's point-plane incidence theorem (Theorem \ref{VinhInc}) to obtain $E\ll V^2|\U-\V|^2/p+|\U-\V|^2.$ With this bound of $E,$ we have 
\[|S_\chi(\T, \U, \V, \alpha, \beta, f)|\lesssim \left(k^{1/n}\cdot \frac{UVW}{p^{1/2n}}+k^{1/n}\cdot \frac{UVW}{V^{1/n}}\right)\cdot \begin{cases}T^{1/2}p^{1/2} ~~\mbox{if}~~n=1\\
Tp^{1/4n}+T^{1/2}p^{1/2n} ~~\mbox{if} ~~n\ge 2,\end{cases}\]
which completes the proof of the second part of Theorem \ref{thm1}. $\square$

\end{document}